\documentclass[11pt]{article}
\usepackage{epsfig,amsmath,amsfonts,amssymb,amsthm}

\newtheorem{prop}{Proposition}[section]
\newtheorem{thm}[prop]{Theorem}
\newtheorem{lemma}[prop]{Lemma}
\newtheorem{cor}[prop]{Corollary}

\newcommand{\N}{{\mathbb N}}
\newcommand{\Z}{{\mathbb Z}}
\newcommand{\Q}{{\mathbb Q}}

\newcommand{\lam}{\lambda}

\newcommand{\cL}{{\mathcal L}}
\newcommand{\cM}{{\mathcal M}}
\newcommand{\cN}{{\mathcal N}}
\newcommand{\cP}{{\mathcal P}}

\def\pmod#1{\allowbreak\mkern8mu({\rm mod}\,\,#1)}
\def\lrbigparen#1{\bigl( {\, {#1}\, \bigr) } }
\def\lrbigbrack#1{\bigl\{ {\, {#1}\, \bigr\} } }

\newcommand{\lasteq}{\rm{(\theequation)}} 

\newcounter{abci}\renewcommand{\theabci}{\alph{abci}}
\newenvironment{abclist}{\begin{list}{(\theabci)}
{\usecounter{abci}
\setlength{\labelwidth}{2cm}
\setlength{\topsep}{0ex} 
\setlength{\itemsep}{0ex}
\setlength{\parsep}{0ex}}}
{\end{list}}

\title{Inhomogeneous Diophantine approximation of some Hurwitzian numbers}

\author{Richard T. Bumby \\
 Rutgers, the State University of New Jersey, \\ Department of
  Mathematics, \\ Hill Center, Busch Campus, \\ 
110 Frelinghuysen Road, \\
  Piscataway, NJ 08854-8019, USA , \\
{\tt bumby@math.rutgers.edu} 
\and 
Mary E. Flahive \\
 Department of Mathematics, \\ 
Oregon State University, \\ Corvallis, OR 97331-4605, USA \\
{\tt flahive@math.oregonstate.edu} }

\date{{\bf MSC:} 11J70,11J06, 11H50}

\begin{document}

\bibliographystyle{plain}  
\pagenumbering{arabic}
\addtocounter{secnumdepth}{1}

\maketitle

\begin{abstract} We continue the work of Takao Komatsu, and consider   
the inhomogeneous approximation constant $L(\theta,\phi)$ for 
Hurwitzian $\theta$ and $\phi \in \Q(\theta) +\Q$. 
The current work uses a compactness theorem to relate 
such inhomogeneous constants to the homogeneous approximation constants. 
Among the new results are: a characterization of such 
pairs $\theta,\phi$ for which $L(\theta,\phi)=0$, 
consideration of small values of $n^2 \, L(e^{2/s},\phi)$
for  $\phi=(r \theta +m)/n$, and the proof of a conjecture 
of Komatsu.  
\end{abstract}

\section{Introduction}    The \emph{inhomogeneous approximation constant}
 for a pair of real numbers~$\theta,\phi$ 
(with \mbox{$\phi \notin \Z\theta+ \Z$)} is 
\[ L(\theta, \phi)=
\liminf_{|q| \to \infty}\, \Big\{ |q|\,  \| q \, \theta - \phi \| \,: 
\, q \in \Z \Big\} , \] 
where 
$\|x\|$ denotes the distance from the real number~$x$ to the nearest integer.
Minkowski proved that  when $\theta$ is irrational, 
$ L(\theta, \phi) \le 1/4$ holds for 
all~$\phi$. Grace~\cite{gr1918} 
used regular simple continued fractions to construct $\theta$ with 
\mbox{$ L(\theta, 1/2) = 1/4$.}  Further historical details on these and 
related results can be found in Koksma~\cite{koksma}. 
In the middle of the twentieth century
there was substantial work related to 
these inhomogeneous approximation constants and also to the associated
inhomogeneous Markoff values.  Reference~\cite{cus94} contains 
a good overview of this 
work and has a comprehensive list of references.

In the last decade, interest in these problems was rekindled by the authors 
of~\cite{cus93,cus96, cus94}, 
and continued with articles by Christopher Pinner~\cite{pinner01} and 
Takao Komatsu~\cite{tk1997,tk1999,tk1999a, tk2002}.  In particular, 
Komatsu used  
several different types of continued fractions to compute 
the inhomogeneous constants when  $e^{1/s}$ (for positive integer~$s$) 
is paired with various~$\phi$ in~$\Q \, \theta + \Q$.  
In this article we make use of the ``relative rationality''  
of these pairs $\theta, \phi$ 
to show how the technically simpler ideas of Grace~\cite{gr1918}
and regular simple continued fractions can be used to unify and 
extend Komatsu's results.

Perron~\cite[Section~32]{OP} defines  
an \emph{arithmetic progression of order~$m$} to be a polynomial
of degree~$m$ with rational coefficients that is a function from~$\N$ 
to~$\N$.
The real number~$\theta$ is a \emph{Hurwitzian number of order~$m$}
if there exists a finite number of 
arithmetic progressions $f_1(x), \ldots, f_R(x)$  
of order at most $m$ (and at least one has order~$m$) such that 
\begin{equation*}
 \theta=[b_0;b_1,\ldots, b_n, f_1(1), \ldots, f_R(1),
f_1(2), \ldots, f_R(2),\ldots ] \; .
\end{equation*} 
We use Perron's convenient notation
\begin{equation*}
\theta =[\, b_0;b_1,\ldots, b_n, 
\, \lrbigparen{\, f_1(i), \ldots, f_R(i)}_{_{i=1}}^{^ \infty} \,] \; .
\end{equation*}

Quadratic irrationals are the Hurwitzian numbers of order~0. 
 For a nonzero integer~$k$, \, $e^{2/k}$ and $\tanh(1/k)$ 
are examples of Hurwitzian numbers of order~1. 
In 1714 Roger Cotes found the continued fraction
expansion of $e$:
\begin{equation*}
e= [\, 2;1,2,1,1,4,1,1,6, \ldots \, ]
=[\, 2;\lrbigparen{1 \, , \, 2j \, ,\, 1 }_{\;_{j=1}}^{^\infty}\, ] \; . 
\end{equation*}
 Euler~(1737) proved this is indeed the continued fraction 
of~$e$, and also that for integers~$s \ge 2$,
\begin{equation}\label{E:Euler}
e^{1/s}
=[\,1; \, 
\lrbigparen{ (2j-1)s -1\,  ,\,  1 ,\,  1 \; }_{\; _{j=1}}^{^ \infty} \, ]
\end{equation}
and 
\[ \tanh(1/s)
=[\,0; \, \lrbigparen{ (2j-1)s} _{ _{j=1}}^{^ \infty} \, ] \; .\]
In correspondence with Hermite, Stieljes 
described the continued fraction of $e^{2/k}$ for odd $k$: 
\[e^2=[\, 7;\, \lrbigparen{3j-1,1, 1, 3j,12j+6}_{\,_{j=1}}^{\; ^\infty} \,] 
\; , \]
and  for integers~$s \ge 1$,
\small 
\begin{equation}\label{E:Stieljes}
 e^{2/(2s+1)}=[\, 1 ; \, \lrbigparen{3(2s+1)j+s,6(2s+1)(2j+1), 
3(2s+1)j+5s+2,1, 1}_{\,_{ j=0}} ^{\; ^\infty} \, ]\; .
\end{equation}
\normalsize
\noindent  For references and insight into the proofs, 
we refer the reader 
to~\cite{cohn2006,osler2006,OP}, with an additional 
comment on the continued fraction of \mbox{$\alpha=\coth(1/s)$.} 
Since $\alpha$ is the result of applying a linear fractional 
transformation with integer coefficients to $\beta=e^{2/s}$, an 
algorithm of G. N. Raney~\cite{raney73} 
(also reported in~\cite{benyon83}) can be used to relate the continued 
fractions of $\alpha$ and $\beta$.  
  
Here we restrict to $\phi \in \Q \theta + \Q$ 
(where $\phi \notin \Z \theta + \Z$).  
By definition, $L(\theta,\phi_1)=L(\theta,\phi_2)$  when  
 $\phi_1-\phi_2 \in \Z \, \theta +\Z$, and so it suffices to assume that 
$\phi$ is in \emph{reduced form}:
\begin{equation}\label{E:redform} 
\phi=\frac{r\theta +m}{n}  \quad \mbox{ and } n \ge 2 \; , \; 
\gcd(r,m,n)=1 \; 
\mbox{ and } \; 0 \le r,m < n \, . 
\end{equation}
The integer~$n$ will 
be called the \emph{reduced denominator} of~$\phi$.

\section{Connections with homogeneous approximation}\label{S:connections}
In this section we consider~$\theta$ of the form 
\begin{equation}\label{E:genform}
\begin{aligned}
\theta = &[ c_0; c_1 \ldots, c_{n_1},a_0, c_{n_1+1}, \ldots, c_{n_1+n_2},a_1,
c_{n_1+n_2+1}, \ldots] \, ,   \\ & 
\mbox{where $ \lim_{i \to \infty}{a_i}=\infty $, $n_j \ge 0$, 
  and $\{ c_i \}$ is a bounded sequence}\, .
\end{aligned} 
\end{equation}
We use standard results on simple continued 
fractions that can be found for example 
in~\cite{cusickandflahive89,Lang95,OP,RandS}.  Our 
principal reference \mbox{is~\cite[Chapter 1]{Lang95}.}  

Let $\theta=[b_0;b_1, b_2, \ldots ]$ be the simple continued fraction
of the real number~$\theta$.  For~$i \ge 0$, \; $\cP_i=(p_i,q_i)$ 
is called the $i$-th \emph{convergent} 
of~$\theta$ if $q_i > 0$ and $[b_0;b_1, b_2, \ldots, b_i ]=p_i/q_i$ in   
reduced form. Then $\cP_0=(b_0,1)$, and using $\cP_{-1}=(1,0)$ we have 
\begin{equation}\label{E:cvgtrec}
\cP_{i+1}=b_{i+1} \cP_i+\cP_{i-1} \; 
\mbox{ for all~$i \ge 0$}  \, , 
\end{equation}  
and 
\begin{equation} \label{E:leap}
 q_i \,|\, q_i \theta -p_i\,|=q_i\, \|\, q_i \theta \, \| = 
\mu_i^{-1}, \;  \mbox{where 
$\mu_i:=[b_{i+1};b_{i+2}, \ldots]+[0;b_i, \ldots, b_1]$} \, . 
\end{equation} 
(Refer to Theorem~1 in~\cite[page~2]{Lang95} and 
Corollary~3 in~\cite[page~5]{Lang95}.)

 For~$\theta$ of the form in~(\ref{E:genform}), 
 the subscripts~$I_j$ 
for which  $b_{_{I_j+1}}=a_j$ will be referred to as 
\emph{leaping subscripts}  with 
associated \emph{leapers} 
\mbox{$\cL_j=(p_{_{I_j}},q_{_{I_j}})$}. The name is appropriate 
since from~(\ref{E:leap}) the rational number given by 
a leaper yields a 
very efficient rational approximation to~$\theta$ as compared with 
the approximations using earlier convergents. This terminology was used 
 by Komatsu in~\cite{tk2003} in a slightly different 
context.

\begin{thm}\label{T:coarse}
Let $\phi=(r \theta+m)/n$ 
be in reduced form. If there exists an integer $g$ such 
that 
\begin{equation} \label{E:congr}
g \cP_i \equiv (m, -r) \pmod{n}
\end{equation} holds for infinitely many 
convergents of~$\theta$ then  
\[ n^2 \, L(\theta, \phi) \le 
g^2 \lrbigparen{\limsup_{i \to \infty} \{ \mu_i \, : \, 
g \cP_{i} \equiv (m, -r) \pmod{n} \, \}} ^{-1} \; .  \]
Moreover, if (\ref{E:congr}) holds for infinitely many leapers, 
then $ L(\theta, \phi)=0$.
\end{thm}
\begin{proof} Let $\{i_j\}$ be the infinite sequence for which 
$g\cP_{i_j} \equiv (m, -r) \pmod{n}$.  Then for each $j$ there exist
integers $R_j,S_j$ such that $g \cP_{i_j}=(m+nR_j,nS_j-r)$;
\begin{align*}
 n^2|(S_j-r/n)(S_j \theta - \phi -R_j)| &=
|nS_j-r|\,|(nS_j-r) \theta -(m+R_jn)| \\
&=g^2 |q_{i_j}| |q_{i_j} \theta -p_{i_j} | \, ;
\end{align*} 
by (\ref{E:leap}),  
\begin{equation}\label{E:lam}
n^2|(S_j-r/n)(S_j \theta - \phi -R_j)| = g^2/\mu_{i_j} \, . 
\end{equation} 
Therefore, 
\begin{align*}
n^2 L(\theta,\phi)&=
n^2 \liminf_{|q| \to \infty}\, \{ |q|\,  \| q \, \theta - \phi \| \,: 
\, q \in \Z \}\\
 &\le n^2 \liminf_{j \to \infty} \{ |S_j|\,| S_j\, \theta - \phi -R_j| \, \} \\
 &= n^2 \liminf_{j \to \infty} \{ |S_j -r/n|\,  
| S_j \, \theta - \phi -R_j| \,\} \\
&=g^2 \, \liminf_{j \to \infty} \{ 1/\mu_{i_j}\, \}\, .
\end{align*}  Since (\ref{E:congr}) is a congruence modulo~$n$, 
we may assume $1 \le g < n$, giving 
\[L(\theta,\phi) \le  \liminf_{j \to \infty} \{ 1/\mu_{i_j} \} 
 \le  \liminf_{j \to \infty} \{ 1/b_{i_j} \} \, . \]
Therefore,  $L(\theta,\phi)=0$
when there are infinitely many leapers satisfying~(\ref{E:congr}).  
 \end{proof}

Theorem~\ref{T:coarse} was implicit in Grace's work~\cite{gr1918}. 
We illustrate its usefulness by proving that for any integer $k \ge 3$,  
\; $L(e^{2/k},(e^{2/k}+1)/2)=0$.  This was 
proved by Komatsu for even~$k$ in~\cite[Theorem 3.1]{tk1999}.
 From~(\ref{E:Euler}) and (\ref{E:Stieljes}), we note that
the sequence of convergents for~$e^{2/k}$ is completely periodic~modulo~2.
In fact, for odd $k=2s+1$, the modulo~2 sequence of convergents of
$e^{2/k}$ has period 
\begin{equation}\label{E:introper--odd} 
(1,1), (s+1,s), (1,1), (0,1), (1,0),(1,1), (s,s+1), (1,1), (0,1), (1,0) \, ,
\end{equation} 
where the leapers are congruent to $(1,1), (s+1,s), (s,s+1)$ modulo~$2$.  
Since these
are all of the congruence classes modulo~2, \, $L(e^{2/k},\phi)=0$ for 
all $\phi$ whose reduced denominator is~2.  On the other hand,
for even $k=2s$ the modulo~2 period for the convergents 
of $\theta=e^{1/s}$ is 
\begin{equation}\label{E:introper} 
(1,1), (s,s+1), (s+1,s), (1,1), (0,1), (1,0) \, , 
\end{equation}
where every leaper is congruent to $(1,1) \pmod{2}$.  Again
$L(\theta, \phi)=0$ for \mbox{$\phi= (\theta+1)/2$.}  But since
$\lim_{i \to \infty}{ \mu_{6i+4}}=\lim_{i \to \infty}{ \mu_{6i+5}}=2$, 
applying Theorem~\ref{T:coarse} with $g=1$ gives $L(e^{1/s}, \phi) \le 1/8$  
for each of $\phi=1/2, -\theta/2$.

\medskip 
\begin{lemma}\label{L:lem} Let $\theta=[b_0;b_1, b_2, \ldots ]$ be irrational  
and $ \phi=(r\theta +m)/n$ be in reduced form. 
For any nonzero integer~$S$, set 
\[\lam(S):= \Bigl| S-\frac{r}{n} \Bigr| \, \|S \theta -\phi \| \, ,  \]
and let $R$ be the nearest integer to $S\theta - \phi$.  
If $0 <n^2 \, \lam(S) <1$ 
then there exist integers $i, g$ with $g$ invertible modulo~$n$ 
such that either  
\begin{equation}\label{E:1alt} (m+Rn, Sn-r)=g \cP_i  
\quad \mbox{and} \quad 
n^2 \lam(S)=\frac{g^2}{\mu_i} 
\end{equation}
or \; 
\begin{equation} \label{E:2alt}
b_{i+1} \ne 1 \; \mbox{ and } \; 
n^2 \lam(S) \ge g^2(1 - w) \; \; 
\mbox{for some  $0 \le w \le [0;b_{i+1}]$} \, .
\end{equation}
Moreover, if $n^2 \, \lam(S) <1/2$, then~(\ref{E:1alt}) must hold.  
\end{lemma}
\begin{proof} Define the integers $M:=m+Rn$ and $N:=Sn-r$.  Then  
calculation gives 
\begin{equation*}|N| \,  |N\theta -M|=n^2 \lam(S)\, .
\end{equation*} Since $0< n^2 \lam(S) < 1$, then 
$N \ne 0$ and $M/N$ is a rational that satisfies 
\begin{equation*} \Bigl| \theta -\frac{M}{N} \Bigr| < \frac{1}{N^2} \, .
\end{equation*} By Theorem 10 
in~\cite[page~16]{Lang95},  there exist integers $i,g$ such that 
either $(M,N)=g\cP_i$ or $b_{i+1} \ne 1$ and 
$(M,N)=g(d\cP_i+\cP_{i-1})$ where 
$d$ equals~$1$ or $b_{i+1}-1$.  In either case, $\gcd(g,n)$ must 
divide both $M$ and $N$, and so each of $r,m$.  
The fact that $\phi$ is reduced therefore
implies $g$ is invertible modulo~$n$.  
In addition, 
 by  Corollary~2 in~\cite[page~11]{Lang95}, if $n^2 \, \lam(S) < 1/2$
then  $(M,N)=g\cP_i$.  Also, if $(M,N)=g \cP_i$ then (\ref{E:lam}) yields 
$n^2 \lam(S)=g^2/\mu_i$, which is~(\ref{E:1alt}). It remains to prove 
$(M,N)=g(d\cP_i+\cP_{i-1})$ implies~(\ref{E:2alt}).

We note that $d=b_{i+1}-1$ gives 
$d \, \cP_i+\cP_{i-1} =\cP_{i+1} - \cP_i$, and so  
the two possibilities can be combined as $(M,N)=g \,(\cP_j \pm \cP_{j-1})$ 
for $j=i,i+1$ where the upper sign is taken when $j=i$ and the lower 
sign when $j=i+1$.  Therefore, 
\begin{align*} 
n^2 \lam(S)&=|N| \, |N \, \theta -M| \\
&=g^2 (q_{j} \pm q_{j-1}) \, 
|(q_{j} \pm q_{j-1})\, \theta -(p_{j} \pm p_{j-1})| \\
&=g^2 (q_{j} \pm q_{j-1}) \, 
|(q_{j-1}\, \theta -p_{j-1})\pm (q_{j}\, \theta -p_{j})| \\
&=g^2 (q_{j} \pm q_{j-1})
\Bigl(\|q_{j-1}\, \theta \| \mp  \|q_{j}\, \theta \|\Bigr) 
\, ,
\end{align*}
since the differences $q_{k}\, \theta -p_{k}$ 
alternate in sign.  Then~(\ref{E:leap}) implies   
\begin{equation}
n^2 \lam(S)
=g^2\Bigl( 1 \pm \frac{q_{j-1}}{q_{j}} \Bigr) \, 
\Bigl(\frac{q_{j}}{q_{j-1}} \frac{1}{\mu_{j-1}} \, \mp
 \, \frac{1}{\mu_{j}} \Bigr) \, .
\end{equation}   
 For $x:=[0;b_j \ldots, b_1]$ and 
$y:=[0;b_{j+1},b_{j+2}, \ldots]$, 
\[ \mu_{j-1} =
 \frac{1}{x}+y \quad 
\mbox{,} \quad \mu_{j} = \frac{1}{y}+x \, ,   \]
and $q_{j-1}/q_j=x$ by  Theorem~4 in~\cite[page~6]{Lang95}.
Putting these
into \lasteq \, yields
\[ n^2 \lam(S)=g^2(1 \pm x)\Bigl(\frac{1}{1+xy} \mp \frac{y}{1+xy} \Bigr)
=g^2 \frac{(1 \pm x)(1 \mp y)}{1+xy} \, \]
where $0 \le x \le [0;b_j]$ and $0 < y < [0;b_{j+1}]$.  
When the upper sign holds (that is, \mbox{when~$j=i$),} 
\mbox{$x \ge 0$} and $ y\le 1 $ yield $ n^2 \lam(S) \ge g^2(1 - y)$ 
and $w=y$ satisfies conclusion~(\ref{E:2alt}).  Analogously, 
$w=x$ can be used for the lower sign.  
\end{proof}

\begin{thm}\label{T:basic}  
Let $\theta$ be as 
in~(\ref{E:genform})  and $\phi=(r \theta+m)/n$ be in 
reduced form. 
Then $L(\theta, \phi)=0$ if and only if there exist infinitely many 
leapers $\cL_j$ such that  
$g_j \, \cL_j \equiv (m, -r) \pmod{n}$
for an integer $ g_j $ that is invertible~modulo~${n}$.
\end{thm}
\begin{proof} 
%
%
Let $\{S_k\}$ be an infinite sequence of nonzero integers such that
\[L(\theta, \phi) 
=\lim_{k \to \infty} \, |S_k| \, \| S_k \, \theta - \phi \| \, . \]
If $L(\theta, \phi)=0$, then 
\[ 0=L(\theta, \phi) =\lim_{k \to \infty} 
\, \Bigl|S_k -\frac{r}{n} \Bigr| \, \| S_k \, \theta - \phi \|
=\lim_{k \to \infty} {\lam(S_k)} \, \, . \] 
Restricting to $k$ satisfying $n^2 \, \lam(S_k) <1/2$,  
for each such $k$ Lemma~\ref{L:lem} 
implies there exist $i_k$ and invertible $g_k$ modulo~$n$ such that
(\ref{E:1alt}) holds.  Then 
\[ b_{i_k+1} +2 \ge \mu_{i_k} =
\frac{g_k^2}{n^2 \, \lam(S_k)} \longrightarrow \infty \, . \]
The condition on $\{c_i\}$ in~(\ref{E:genform}) implies  
$i_k$ is a leaping subscript for sufficiently 
large~$k$.

The converse was proved in  Theorem~\ref{T:coarse}
\end{proof}
\begin{cor}\label{C:all} Let $\theta$ be as 
in~(\ref{E:genform})  and $\phi=(r \theta+m)/n$ be in 
reduced form.  If 
$L(\theta, \phi)=0$ then $L(\theta, g \, \phi)=0$
for every integer $g$ that is not a multiple of $n$. 
 In particular, for any $n \ge 2$ there exists $m/n$ such 
that $L(\theta, m/n)=0$ if and only if $L(\theta, m_1/n)=0$ 
for all $m_1 \in \Z$, $m_1/n \notin \Z$.  
\end{cor}
\begin{proof} Let  $g$ be an integer that is not a multiple of~$n$. 
By Theorem~\ref{T:basic}, 
$L(\theta,\phi)=0$ implies there exist 
infinitely many leapers~$\cL_{j_k}$ such that 
$g_k \, \cL_{j_k} \equiv (m,-r) \pmod{n}$ for some invertible 
$g_k \pmod{n}$, and so 
\begin{equation*}  g_k \, g \, \cL_{j_k} \equiv (gm,-gr) \pmod{n} \,  \, 
\mbox{ for all } k \; .  
\end{equation*}
Setting $d:=\gcd(g,n)$ and $h:=g/d$ this implies 
\[ g_k \, h \, \cL_{j_k} \equiv  (hm,-hr) \pmod{n/d} \,  \, 
\mbox{ for all } k \; .\]
Since $g_kh$ is invertible modulo~$n/d$, from Theorem~\ref{T:basic} 
we obtain \mbox{$L(\theta, g \, \phi)=0$.}
\end{proof}

Henceforth, we'll restrict consideration to a slight generalization 
of~$e^{2/k}$; namely, 
\begin{equation}\label{E:genform2} 
\begin{aligned}
\theta = [a_0;& c_1, \ldots, c_{n_1},a_1,c_{n_1+1}, 
\ldots, c_{n_1+n_2},a_2, \ldots] \, ,  
\mbox{where $\lim_{i \to \infty}{a_i}=\infty$} \\ 
& \mbox{and either~$\{ c_i \}$ is a finite sequence 
or $\limsup_{i \to \infty}{c_i}=1$}\, . 
\end{aligned}
\end{equation}

\begin{thm}\label{T:small} 
Let $\theta=[b_0;b_1,b_2, \ldots]$ be as 
in~(\ref{E:genform2}) and $\phi=(r \theta+m)/n$ be in 
reduced form.  If  $0 < n^2 \, L(\theta, \phi) < 1$,  
then there exist infinitely many 
non-leaping convergents~$\cP_i \equiv (m, -r) \pmod{n}$, and 
\begin{equation}\label{E:positivevalue} 
n^2 \, L(\theta, \phi) = 
\lrbigparen{\limsup_{i \to \infty} \{ \mu_i \, : \, 
\cP_{i} \equiv (m, -r) \pmod{n} \, \}}^{-1} \; . 
\end{equation}
\end{thm}
\begin{proof} From~(\ref{E:genform2}), 
there exists~$I$ such that for~$i \ge I$,
\begin{equation} 
b_{i+1} \ne 1 \iff \mbox{$i$ is a leaping subscript} \, .
\end{equation} 
Let $\{S_j\}$ be an infinite sequence 
of nonzero integers such that
\[  L(\theta, \phi) =\lim_{j \to \infty} {\lam(S_j)} \, \, . \]
  From $0< n^2 \, L(\theta, \phi) < 1$ it follows that 
$0< n^2 \, \lam(S_j) <1 $ holds for infinitely many~$j$ 
and Lemma~\ref{L:lem} can be applied:  For each such~$S_j$, 
we obtain a subscript~$i=i_j$ such that one of the conclusions 
of the lemma holds. By \lasteq, if  $\lam(S_j)$ 
satisfies~(\ref{E:2alt}) for sufficiently 
large $j$, then  the subscript~$i_j$ must be leaping.
If an infinite subsequence~$S_{j}$ were to satisfy~(\ref{E:2alt})
with leaping subscript $i_{j}$ and associated $w_j$, 
then 
\[n^2 \lam(S_j) \ge g_j^2(1 - w_j) \quad 
\mbox{where $w_j \le [0;b_{i_j+1}] \longrightarrow 0$}\, ,   \] 
and we would obtain the contradiction  
\[n^2 \, L(\theta,\phi) \ge \lim_{j \to \infty} {g_j^2} \ge 1 \, . \]
(This is similar to the argument in~\cite[p. 116]{RandS}.)
Therefore, for sufficiently large $j$, \, $\lam(S_j)$ 
satisfies (\ref{E:1alt}) for some $i=i_j$ that is not 
leaping --- else $L(\theta,\phi)$ would be zero.  Since 
$\limsup_{i \to \infty}{c_i}=1$, then  $ \mu_{i_j} \le 3$ for 
all but finitely many~$j$, and 
\[ \lim_{j \to \infty}{\frac{g_j^2}{3}} \le
\lim_{j \to \infty}{\frac{g_j^2}{\mu_{i_j}}} 
= n^2 \, L(\theta,\phi) < 1  \, , \]  which implies 
$g_j=1$ for all sufficiently large~$j$.  Therefore, 
$\cP_{i_j}\equiv (m, -r) \pmod{n}$ for infinitely many non-leaping~$i_j$ 
and also~(\ref{E:positivevalue}) holds.
\end{proof}

The hypothesis $L(\theta, \phi)>0$ in  Theorem~\ref{T:small}
guarantees that 
at most finitely many leapers are congruent to $(m, -r) \pmod{n}$.
It's worth noting that $n^2 \, L(\theta, \phi) < 1$  
implies the existence of infinitely many convergents 
$\cP_i \equiv g \, (m, -r) \pmod{n}$ with $g=1$.

We return to the earlier question 
of calculating $L(e^{1/s}, \phi)$ for  
$\phi$ whose reduced denominator equals $2$.  Recall 
the sequence of convergents of $\theta$ is completely periodic 
modulo~2 with period given in~(\ref{E:introper}). Since  
$\lim_{j \to \infty}{\mu_{6j+k}}=2$ for all 
\mbox{$k \not \equiv 0 \pmod{3}$,}  
application of Theorem~\ref{T:small} gives 
 $L(e^{1/s}, \phi)= 1/8$ for $\phi=1/2, e^{1/s}/2$. 

\section{Komatsu's Conjecture} \label{S:KConjecture}\
In Theorem~\ref{T:conjecture} 
of this section we prove a generalization of the 
conjecture of T.~Komatsu~\cite[p. 241]{tk2002} that 
for integers $s \ge 1$, \; 
$n^2 \,  L(e^{1/s},1/n) = 0 \; \mbox{or} \; 1/2$  
for all $n \ge 2$.

\begin{prop}\label{P:per2} Let $k, n$ be positive integers with $n\ge 2$.  
 For any sequence of integers~$\{ b_j \}$, define a 
sequence $\{s_j\} \subset \Z^k$ inductively 
using any initial values $s_0,s_1 \in \Z^k$, and 
\begin{equation} s_{j} \equiv b_js_{j-1}+s_{j-2} \pmod{n} \,  
\mbox{ for all $ j \ge 2$} \; . 
\end{equation}
If $\{b_j\}$ is periodic modulo~$n$, then 
 $\{s_j\}$ is periodic. If $\{b_j\}$ is completely periodic modulo~$n$, then 
 $\{s_j\}$ is  also completely periodic. 
\end{prop}
\begin{proof} 
If~$b_{i+1}, \ldots, b_{i+t}$ is a period for~$\{ b_j \}$, consider 
the following sequence of pairs:
\[(s_{i-1},s_{i}),(s_{i-1+t},s_{i+t}),(s_{i-1+2t},s_{i+2t}), 
\ldots  \, .\]
This infinite sequence eventually has a repetition 
modulo~$n$.
Because~$\{s_j\}$  satisfies the recurrence~\lasteq \, 
and~$b_{i+1}, \ldots, b_{i+t}$ is a period for~$\{ b_j \}$,
the first repetition in this sequence 
will identify the beginning of a period for~$\{s_j\}$.
Since $s_{i-1}$ is determined by 
$s_{i-1} \equiv s_{i+1} -b_{i+1}s_i \pmod n$,
when~$\{ b_j \}$ is completely periodic modulo~$n$,  
$\{s_j\}$~must also be completely periodic. 
\end{proof}

In particular, since the partial quotient sequence of 
every Hurwitzian number is periodic modulo every 
integer~$n \ge 2$, its sequence of convergents is periodic modulo~$n$.

 Henceforth. for $\theta$ of the form in~(\ref{E:genform2})  
we further restrict to $n \ge 2$ for which the 
partial quotient sequence of~$\theta$ is periodic modulo~$n$. 
If~$b_{I+1}, \ldots,b_{I+T}$ is a period for the partial 
quotient sequence such that 
$\cP_{I+1}, \ldots,\cP_{I+T}$ is a period for the 
convergents, 
we define $M_j:= \limsup_{k \to \infty} \mu_{j+kT}$ 
for all $j=I+1, \ldots, I+T$ and observe that 
\begin{equation} 
M_j=\infty \iff \limsup_{k \to \infty}{b_{j+kT+1}} = \infty
\iff \mbox{$\cP_{j+kT}$ is a leaper for infinitely many $k$} \,.
\end{equation}

\begin{thm}\label{T:per1} Let $\theta$ have the form given 
in~(\ref{E:genform2}).  Let $n \ge 2$ be such that 
the partial quotient sequence of $\theta$ is periodic modulo~$n$, 
and~$b_{I+1}, \ldots,b_{I+T}$ and $\cP_{I+1}, \ldots,\cP_{I+T}$ 
be as set up above.  
If  $m,r$ are integers with $\gcd(m,r,n)=1$ for which 
there exists $i > I$ with~$\cP_{i} \equiv (m, -r) \pmod{n}$, 
we  set 
\[M:=\max\lrbigbrack{M_j \, : \, I+1 \le j \le I+T  \mbox{ and } 
\cP_j \equiv (m,-r) \pmod{n} } \, . \]
  If $M \ne 1$ then 
$ n^2 \, L(\theta,(m+r\, \theta)/n)= M^{-1}$.
\end{thm}
\begin{proof} Set $\phi:=(m+r\, \theta)/n$, \,  and let 
$j$ be such that $1 \le j \le T$, 
$\cP_j \equiv (m,-r) \pmod{n}$, and $M_j=M \ne 1$. 
The observation in \lasteq \, combined with Theorem~\ref{T:basic} 
gives the 
conclusion for $M=\infty$.  We may therefore assume $M$ is finite, and 
that by \lasteq \, at most finitely many  $\cP_j \equiv (m,-r) \pmod{n}$ 
are leapers, that in turn gives $L(\theta,\phi) > 0$.  
Since 
\[ (p_{j+kT},q_{j+kT})=\cP_{j+kT} \equiv \cP_j \equiv (m, -r) \pmod{n} \; \; 
\mbox{for all $k \ge 0$} \, , \]
 then  
\[n^2 \, \lam(S_k) = \frac{1}{\mu_{j+kT}} \; \; 
\mbox{for $S_k:=(q_{j+kT}+r)/n$} \, . \]  
By Theorem~\ref{T:coarse}, 
\[
0 < n^2\, L(\theta,\phi) \le n^2 \, \liminf_{k \to \infty}\, {\lam(S_k)} 
=\liminf_{k \to \infty}\, {\frac{1}{\mu_{j+kT}}} =\frac{1}{M} < 1 \, , \]
and the conclusion follows from Theorem~\ref{T:small}.
\end{proof}

\begin{thm}\label{T:conjecture}
{\rm [A generalization of Komatsu's conjecture]}  
Let $\theta$ be an irrational whose 
continued fraction  has the form given 
in~(\ref{E:genform2}), and 
let  $n \ge 2$ be such that the 
partial quotient sequence of $\theta$
is completely periodic modulo $n$.  If each  $n_i \in \{ 0,2\}$ then 
\[n^2 \, L(\theta,\phi) \in \{0, 1/2\} \; \;  
\mbox{for both $\phi=1/n$, $\phi=-\theta/n$}  \, .\]
In particular, for every $k \ge 2$ and every $n \ge 2$
\[ n^2 \, L(e^{2/k},\phi) \in \{0, 1/2 \}\; \;  
\mbox{for both $\phi=1/n, \; - \, e^{2/k}/n$} \, . \]
\end{thm}  
\begin{proof} The fact that $n_i \in \{ 0,2\}$  implies 
every $M_j$ equals $\infty$ or $2$.  
By Proposition~\ref{P:per2}, the sequence of 
convergents of~$\theta$ is completely periodic modulo~$n$.  
If $T$ is a period length, 
then
\[ \cP_{T-1} \equiv \cP_{-1} =
(1,0) \pmod{m} \quad \mbox{and} \quad \cP_{T-2} 
\equiv \cP_{-2} =(0,1)  \pmod{m} \, .\]  Since $M \in \{ \infty, 2\}$,
 the conclusion follows from~Theorem~\ref{T:per1}.
\end{proof}

\begin{thm}\label{T:ecor} Let $	k \ge 1$,   $n \ge 2$.  
If $\gcd(n,k) \ne 1$  then  
\[n^2 \, L(e^{2/k},1/n)= n^2 \, L(e^{2/k},-e^{2/k}/n)=1/2 \; . \] 
\end{thm}
\begin{proof} By Theorems~\ref{T:basic} 
and~\ref{T:conjecture} it suffices to prove 
that no component of any leaper of~$\theta=e^{2/k}$ is divisible by~$n$.
Since~$\gcd(n,k) \ne 1$ we first consider all sequences~modulo~$k$.  

When $k=2s$,
the  partial quotient sequence is completely 
periodic modulo~$k$ with period
$1,s-1,1$, and the period of the sequence of convergents is
\[ (1,1) \, , \, (s, s-1) \, , \, (s+1, s) \, , \, (1,-1) 
\, , \, (0, 1)\, , \, (1,0) \, , \]
where the leapers are $(1,\pm 1) \pmod{k}$.
When $k=2s+1$,  then the partial quotient sequence has 
period~$1,s,0,s,1$~modulo~$k$, and the sequence of convergents has period
\small
\[ (1,1) \, , \, (-s, s) \, , \, (1, 1) \, , \, (0,-1) 
\, , \, (1,0) \, , 
\, (1,-1) \, , \, (-s, -s) \,  , \, (1,-1) 
\, , \, (0, 1)\, , \, (1,0) \, , \]
 \normalsize
where the leapers are either $(1,\pm 1)$ or 
$(-s,\pm s) \pmod{k}$.  
In each case we have shown that each component of
every leaper is relatively prime 
to~$k$, and therefore cannot 
be divisible by~$n$.  The conclusion follows from 
Theorems~\ref{T:basic} and~\ref{T:conjecture}.   
\end{proof}

Earlier we proved $L(e^{1/s}, \phi)=1/8$ 
for each of $\phi=1/2, \, e^{1/s}/2$, a special case 
of the last result.  The theorem also generalizes~\cite[Theorem~3]{tk2002}, 
that \mbox{$n^2 \, L(e^{1/s},1/n)= 1/2$} when $n$ divides $s$.

\section{When is $\mathbf{L(e^{1/s},\phi)}$ zero?}
%
%
\begin{thm}\label{T:leapingper} Let 
$s, n $ be positive integers with $n \ge 2$, and 
let~\mbox{$\cL_i=\cP_{3i}=(P_i,Q_i)$} be the $i$-th leaper of~$e^{1/s}$.  

\begin{abclist}
\item Then~$\{\cL_i\}$ is a completely periodic sequence 
modulo~$n$ with period  
\begin{equation}\label{E:period} 
\cL_0, \ldots ,\cL_{K-1},\cL_K, \cL_{K-1}, \ldots, \cL_0, 
\cL^*_0, \ldots ,\cL^*_{K-1},\cL^*_K, \cL^*_{K-1}, \ldots, \cL^*_0 \,
\end{equation}
where $K=\lfloor{ n/2 }\rfloor$ and $\cL^*_i:=(P_i,-Q_i)$. 
\item If~$\gcd(n,2s)=1$, then (\ref{E:period}) is a minimal period 
for the leapers of ~$e^{1/s}$  modulo~$n$.
\item For all $1 \le s <n$, the $i$-th leaper of $e^{1/(n-s)}$ is 
$ (-1)^i(Q_i,P_i) \pmod{n}$\, .
\end{abclist}
\end{thm} 
\begin{proof} 
Perron~\cite[Section 31]{OP} proved that for 
$\theta= [a_0;c_1,c_2,a_1, \ldots, a_i,c_1,c_2,a_{i+1},\ldots]$  
the subsequence $\cP_2, \cP_5, \ldots, \cP_{3i+2}, \ldots$ 
of convergents of~$\theta$ satisfies the second-order recurrence
\[\cP_{3i+2}=(a_i(c_1c_2+1)+c_1+c_2)\cP_{3i-1} +\cP_{3i -4} \; . \]
Therefore, the sequence of leapers of 
\begin{equation*}
e^{1/s}
=[\,1; \, 
\lrbigparen{ 2sj-(s+1)\,  ,\,  1 ,\,  1}_{\; _{j=1}}^{^ \infty} \, ] 
\end{equation*}  satisfies the recurrence
 \begin{equation} \label{E:recr}
\cL_{-1}=(1,-1) \; , \; \cL_{0}=(1,1)  \; , \; 
\cL_{j+1}= A_j\; \cL_j+ \cL_{j-1} \; ,
\end{equation}
for $k:=2s$ and $A_j:=(2j+1)\, k$, a sequence that is completely 
periodic modulo~$n$. 

Since $A_K =(2K+1)\,k \equiv 0 \pmod{n}$, then 
$\cL_{K+1} \equiv \cL_{K-1} \pmod{n}$.  Also, for all $j \ge 0$, 
\[ A_{K+j} +A_{K-j} =2(2K+1)\,k \equiv 0 \pmod{n}\, , \] 
and an inductive argument using the generating 
recurrence~(\ref{E:recr}) yields 
\begin{equation}
\cL_{K+j} \equiv \cL_{K-j} \pmod{n} \; \; \mbox{for all $j \ge 0$} \, . 
\end{equation}
  In particular, for $j=K, K+1$, 
\[\cL_{n-1} \equiv \cL_{0}=(1,1)= \cL^*_{-1} \pmod{n} \; ; \;   
\cL_{n} \equiv \cL_{-1}=(1,-1)=\cL^*_{0} \pmod{n}\, ;\] 
again using recurrence~(\ref{E:recr}) inductively, 
\[ \cL_{n+j} \equiv \cL^*_{j} \pmod{n} \; \; \mbox{for all $j$} \, . \]
In combination with \lasteq \, this implies~(\ref{E:period}) 
is a period for the leapers modulo~$n$.

 Further, if~$T$ is a period-length of the leapers, then 
\[ A_T \cL_T + \cL_{T-1} = \cL_{T+1} \equiv 
\cL_{1} \equiv A_0 \cL_0 + \cL_{T-1} \pmod{n} \; , \]
implying $(0,0) \equiv (A_T-A_0) \, \cL_0 \equiv 2Tk \, (1,1) \pmod{n}$.
When $\gcd(n,k)=1$, \; $T$ must 
be divisible by~$n$.  The fact that 
\[ \cL_n \equiv \cL^*_0 \not \equiv \cL_0 \pmod{n} \]
proves~(\ref{E:period}) is a 
minimal period for the leapers of $e^{1/s}$.  

It remains to prove~(c).  For this, we define 
$\{\cM_j\}$ to be the sequence $\cM_j:=(Q_j,P_j)$ where 
$(P_j,Q_j)$ is the $j$-th leaper of $e^{1/s}$.  Then $\{\cM_j\}$ 
also satisfies the recurrence~(\ref{E:recr}) with initial values 
$\cM_{-1}=(-1,1),\cM_{0}=(1,1)$, and 
the sequence $\cN_j:=(-1)^j \cM_j$ 
satisfies the recurrence 
\[\cN_{-1}=(1,-1) \; , \; \cN_{0}=(1,1)  \; , \; 
\cN_{j+1}= -A_j\; \cN_j+ \cN_{j-1} \, . \]
Since 
\[ (2j+1)\, 2(n-s) \equiv -(2j+1) \, 2s \equiv -A_j \pmod{n} \, , \]
this is the recurrence for the leapers of $e^{1/(n-s)}$. 
\end{proof}

In 1918, D. N. Lehmer~\cite{lehmer1918} investigated the modulo~$n$ 
period of the convergents for certain Hurwitzian numbers. 
More recently,  C. Elsner~\cite{elsn1999} used generating functions
to prove results on the period length of the~modulo~$n$ sequence of leapers 
of~$e$, and  Takao Komatsu~\cite{tk1997,tk1999,tk1999a, tk2002}
found the period length of the  modulo~n leapers
of $e^{1/s}$ always divides $2n$ and it divides $n$ when $n$ is even. 
Both Elsner and Komatsu applied their results to 
homogeneous approximation over congruence classes.

\begin{cor}  Let $s \ge 1$, $n \ge 2$ be integers, and 
define $\theta:=e^{1/s}$.
If $ L(\theta,(m+r \, \theta)/n) =0 $, then 
$ L(\theta,(m-r \, \theta)/n) =0$. 
\end{cor}
\begin{proof}  The conclusion follows from Theorem~\ref{T:basic} 
and the form of the
period in~(\ref{E:period}). 
\end{proof}

\begin{cor}\label{C:specialzeros} Let $n \ge 2$ be a odd integer.
Then for all $m \not \equiv 0 \pmod{n}$, 
\begin{equation} \label{E:counterpart2}
 L(e^{2/(n+1)}, m/n)= 0\quad \mbox{and} \quad 
L(e^{2/(n-1)}, -m \, e^{2/(n-1)}/{n})= 0  \; .
\end{equation}
\end{cor}
\begin{proof} Since $n$ is odd, $s:=(n + 1)/2$ is an integer. 
The first leaper of $e^{1/s}$ can be calculated using 
recurrence~(\ref{E:recr}) with~$k=n+1$:
\[ \cL_1=A_0 \,(1,1)+(1,-1) \equiv 2(1,0) \pmod{n} \, , \]
and from Theorem~\ref{T:leapingper}(c), the first leaper of 
$e^{2/n-1}$ is $-(0,2)$. Therefore, Theorem~\ref{T:basic} 
implies~(\ref{E:counterpart2}) for $m=2$, and the conclusion 
follows from Corollary~\ref{C:all}.
\end{proof}

\begin{thm}\label{T:CRT} Let $s$ be a positive integer.  
If~$n_1,n_2$ are relatively prime integers 
for which $L(e^{1/s},1/n_1)=L(e^{1/s},1/n_2)=0$ then 
\mbox{$L(e^{1/s},1/(n_1 \, n_2))=0$.}
\end{thm}
\begin{proof} Since $L(e^{1/s},1/n_i)=0$, the form 
of the period of the leapers of~$e^{1/s}$ yields
\emph{even}~$1 \le j_i=2r_i \le 2n_i$ with $Q_{2r_i} =0 \pmod{n_i}$.
Using the  Chinese Remainder Theorem, the system 
$r \equiv r_i \pmod{n_i}$ has a solution $r \pmod{n_1n_2}$, 
and $Q_{2r} \equiv Q_{2r_i} \equiv 0 \pmod{n_i}$ for each~$i$. 
Therefore, we have found a subscript $j=2r$ such that
 $Q_{j} \equiv 0 \pmod{n_1n_2}$, and $L(e^{1/s},1/n_1 \, n_2)=0$ .
\end{proof}

\begin{thm}\label{T:algorithm} Let $s$ be a positive integer and 
let $n \ge 3$ be odd. Then for any reduced $\phi=(m+r \, e^{1/s})/n$ 
it is possible to check 
whether or not $L(e^{1/s},\phi)$  is zero 
in fewer than $n/2$ multiplications  modulo~$n$.  
In fact, if $n$ has $t$ distinct prime 
divisors, the number of operations can be reduced to $n/2^t$ 
multiplications modulo~$n$.
\end{thm}
\begin{proof} The form  
of the period in~(\ref{E:period}) allows one to 
conclude whether or not a leaper has the form $g (m,-r) \pmod{n}$ 
within $n/2$ applications of the recurrence~(\ref{E:recr}).   
Theorem~\ref{T:CRT} reduces the question to checking the 
period modulo each prime power divisor of~$n$.
\end{proof} 

The algorithm implicit 
in the proof of Theorem~\ref{T:algorithm} can be used to verify  
that the following values should 
be added to the list given in~\cite[p. 241]{tk2002} of all 
values of $s$ 
for which $L(e^{1/s},1/n)=0$ (for $n \le 49$):
\medskip
\begin{center}
\begin{tabular} {|l|l|}\hline $n$   & $s \pmod{n}$ \\ \hline 
$23 $ & $12$ \\
$25 $ & $13,23$ \\
$29 $ & $15$ \cr
$43 $ & $25$ \cr
$47 $ & $11,17,33,43$  \\ 
$49 $ & $1,22,46$  \\ \hline 
\end{tabular} 
\end{center}
\noindent In particular, notice that~(\ref{E:counterpart2}) ensures all 
of $(n,s)=(23,12),(25,13), (29,15)$ must be included in the table.


 We thank Takao Komatsu and Iekata Shiokawa 
for making the conference a success.  In particular, we thank Takao for 
arranging the funding that allowed our participation in the 
conference.

\bibliography{bumby}

\end{document}